\documentclass[11pt,reqno, a4paper]{amsart}

\makeatletter
\def\@tocline#1#2#3#4#5#6#7{\relax
  \ifnum #1>\c@tocdepth 
  \else
    \par \addpenalty\@secpenalty\addvspace{#2}%
    \begingroup \hyphenpenalty\@M
    \@ifempty{#4}{%
      \@tempdima\csname r@tocindent\number#1\endcsname\relax
    }{%
      \@tempdima#4\relax
    }%
    \parindent\z@ \leftskip#3\relax \advance\leftskip\@tempdima\relax
    \rightskip\@pnumwidth plus4em \parfillskip-\@pnumwidth
    #5\leavevmode\hskip-\@tempdima
      \ifcase #1
       \or\or \hskip 1em \or \hskip 2em \else \hskip 3em \fi%
      #6\nobreak\relax
    \dotfill\hbox to\@pnumwidth{\@tocpagenum{#7}}\par
    \nobreak
    \endgroup
  \fi}
\makeatother

\usepackage[utf8]{inputenc}
\usepackage{geometry}
\usepackage{amsmath}
\usepackage{amsfonts}
\usepackage{amssymb}
\usepackage{graphicx}
\usepackage{amsthm}
\usepackage{mathtools}
\usepackage{array}
\usepackage{enumerate}
\usepackage{float}
\usepackage{tikz-cd}
\usepackage{todonotes}
\usepackage[colorlinks=true, linktocpage=false]{hyperref}
\usepackage{textcomp}
\usepackage[norefs,nocites,ignoreunlbld]{refcheck}
\usepackage{amsaddr}

\numberwithin{equation}{section}
\theoremstyle{definition}
\newtheorem{thm}{Theorem}[section]

\newtheorem{lemma}[thm]{Lemma}
\newtheorem{cor}[thm]{Corollary}

\newtheorem{remark}[thm]{Remark}

\DeclareMathOperator{\GL}{\mathrm{GL}}
\DeclareMathOperator{\SL}{\mathrm{SL}}
\DeclareMathOperator{\PGL}{\mathrm{PGL}}
\DeclareMathOperator{\PSL}{\mathrm{PSL}}

\DeclareMathOperator{\im}{\mathrm{im}}

\newcommand{\n}{\mathfrak{n}}
\newcommand{\p}{\mathfrak{p}}

\newcommand{\Z}{\mathbb{Z}}

\newcommand{\F}{\mathbb F}

\newcommand{\Q}{\mathbb{Q}}
\newcommand{\R}{\mathbb{R}}
\newcommand{\C}{\mathbb{C}}

\newcommand{\Mod}[1]{\ (\mathrm{mod}\ #1)}
\newcommand{\demph}[1]{\textbf{#1}}

\newcommand{\mf}[1]{\mathfrak{#1}}
\newcommand{\Xbar}{\overline{X}}
\newcommand{\Ybar}{\overline{Y}}

\renewcommand\arraystretch{0.667}

\title{Bianchi period polynomials: Hecke action and congruences}
\author{Lewis Combes}
\address{School of Mathematics and Statistics, University of Sheffield, Sheffield, UK.}

\begin{document}

 \begin{abstract} Let $\Gamma=\PSL_2(\mathcal{O})$ be a Bianchi group associated to one of the five Euclidean imaginary quadratic fields. We show that the space of weight $k$ period polynomials for $\Gamma$ is ``dual'' to the space of weight $k$ modular symbols for $\Gamma$, reflecting the duality between the first and second cohomology groups of the arithmetic group $\Gamma$. Using this result, we describe the action of Hecke operators on the space of period polynomials for $\Gamma$ via the Heilbronn matrices. 

As in the classical case, spaces of Bianchi period polynomials are related to parabolic cohomology of Bianchi groups, and in turn, via the Eichler-Shimura-Harder Isomorphism, to spaces of Bianchi modular forms. In the second part of the paper, we numerically investigate congruences between level 1 Bianchi eigenforms via computer programs which implement the above mentioned Hecke action on spaces of Bianchi period polynomials. Computations with the Hecke action are used to indicate moduli of congruences between the underlying Bianchi forms; we then prove the congruences using the period polynomials. From this we find congruences between genuine Bianchi modular forms and both a base-change Bianchi form and an Eisenstein series. We believe these congruences are the first of their kind in the literature.
 \end{abstract} 
	
\maketitle

\tableofcontents

\section{Introduction}
\subsection{Background}
Let $\Gamma$ denote the Bianchi group $\PSL_2(\mathcal{O})$, where $\mathcal{O}$ is the ring of integers of a Euclidean imaginary quadratic field of class number 1 (see Remark \ref{non_euc} for the motivation behind these restrictions). The Eichler-Shimura-Harder isomorphism allows one to embed the space $S_{k+2}(\Gamma)$ of level $\Gamma$ weight $k+2$ cuspidal Bianchi modular forms into the cohomology of $\Gamma$ in a Hecke equivariant way:
\begin{equation*}
	\Theta: S_{k+2}(\Gamma) \hookrightarrow H^1(\Gamma, V_{k,k}(\C)).
\end{equation*}
Here $V_{k,k}(\C)$ is the standard weight module that we define below. It is well-known that the cohomology classes in the image of $\Theta$ can be represented with a $1$-cocycle that is \textit{parabolic}, that is, it vanishes on all parabolic elements of $\Gamma$. Such a 1-cocycle is determined by its value at the element $S=\left ( \begin{smallmatrix} 0 & -1 \\ 1 & 0 \end{smallmatrix} \right )$ of $\Gamma$. In turn, the map 
\begin{equation*}
	c \mapsto c(S)
\end{equation*}
establishes a linear injection of the space of parabolic cocycles into the weight module $V_{k,k}(\C)$:
\begin{equation*}
	\tau: Z^1_{par}(\Gamma, V_{k,k}(\C)) \hookrightarrow V_{k,k}(\C).
\end{equation*}
We will refer to the subspace of $V_{k,k}(\C)$ given by the image of $\tau$ as the \textit{space of Bianchi period polynomials}. 

While period polynomials for the classical modular group $\PSL_2(\Z)$ have been intensely studied since the 1973 paper of Manin \cite{Manin_1973}, Bianchi period polynomials have received much less attention, appearing for the first time in Karabulut's 2021 paper \cite{Karabulut}, in which they are described explicitly by way of known presentations for $\Gamma$. 

In this paper, we consider spaces of Bianchi period polynomials in the Euclidean cases and prove a duality between the space of Bianchi period polynomials and the cohomology group $H^2(\Gamma,V_{k,k}(\C))$, which in turn lets one define a Hecke action on the polynomials themselves. 

Once the Hecke action is introduced, we implement it on a computer and carry out numerical investigations. Our motivation is to find congruences between Bianchi modular forms by way of inspecting the prime divisors of various quantities coming from their period polynomials. The most famous example of such a phenomenon is the appearance of the prime 691 in the denominator of the leading coefficient of the normalised period polynomial associated to the discriminant newform, $\Delta \in S_{12}(\PSL_2(\Z))$, indicating its congruence with an Eisenstein series. In the spirit of this example, we exhibit congruences between genuine cusp forms of level 1 over $\Q(\sqrt{-11})$ and 
\begin{enumerate}
	\item an Eisenstein series,
	\item the base-change of a cusp form,
\end{enumerate}
modulo (primes over) 173 and 43 respectively. As far as we are aware, these are the first records of such congruences between higher weight Bianchi modular forms in the literature.

\subsection{Data availability statement}
We have written code to compute spaces of period polynomials and the associated Hecke action in \textsf{Magma} \cite{Magma}, which is available on GitHub at
\begin{center}
	\href{https://github.com/lewismcombes/BianchiPeriodPols}{\texttt{https://github.com/lewismcombes/BianchiPeriodPols}}
\end{center} 

The repository also contains many eigenvalues of the above-mentioned forms, numerically witnessing the congruences we prove (see Section 5). The code was written and run on \textsf{Magma} V2.25-4, on the Sheffield School of Mathematics and Statistics server, Intel\textsuperscript{\textregistered} Xeon\textsuperscript{\textregistered} CPU E5-2683 v3 @ 2.00GHz, with 128GB of RAM, running Ubuntu 22.04.2 LTS. The data computed can also be found at
\begin{center}
	\href{https://doi.org/10.6084/m9.figshare.25039679}{\texttt{https://doi.org/10.6084/m9.figshare.25039679}}
\end{center}

\subsection{Structure of the paper}
In Section 2 we recall various facts about the Bianchi groups $\Gamma$ and their cohomology groups. In Section 3 we connect these cohomology groups with period polynomials, proving a duality between period polynomials and modular symbols. In Section 4 we use this duality and Heilbronn matrices to define a Hecke action on period polynomials, and in Section 5, we give two methods that (conjecturally) allow the detection of congruences between Bianchi modular forms from their period polynomials. 

\subsection{Acknowledgements}
This work was completed while in receipt of the Engineering and Physical Sciences Research Council grant EP/R513313/1, as part of the author's Ph.D. studies, and will form part of their thesis. The author would like to thank Haluk \c{S}eng\"{u}n for his manifold guidance and support during the writing of this paper, which would have been much harder and much less interesting without his input. Additionally, some of the code appearing in the repository noted above was written by \c{S}eng\"{u}n (see the document there for more details). The author would also like to thank Alexandru Popa for sharing code to compute period polynomials for classical modular forms, which was instructive in our implementation of the Hecke action. It is our pleasure to thank Chris Williams and John Cremona for helpful comments on an earlier draft, and Aurel Page for pointing out how to prove the congruences we detect. All these comments helped the paper a great deal. Finally, we are grateful to the referee for useful feedback that has improved the paper.

\section{Cohomology of Bianchi groups}

\subsection{Bianchi groups}
Let $K=K_d=\Q(\sqrt{-d})$ be an imaginary quadratic field with ring of integers $\mathcal{O}=\mathcal{O}_d$. The modular groups 
\begin{equation*}
	\Gamma = \Gamma_d:=\PSL_2(\mathcal{O}_d)
\end{equation*}
are known as \demph{Bianchi groups}. We will restrict our attention to the Euclidean Bianchi groups, that is, we will assume that $d \in \{ 1,2,3,7,11 \}$ so that $\mathcal{O}_d$ is Euclidean. We fix the following notation throughout: 
\begin{equation*}
	 \omega =  \omega_d \coloneqq \begin{cases} 
	\sqrt{-1}, \qquad d=1, \\  
	\sqrt{-2}, \qquad d=2, \\
	\frac{1+\sqrt{-3}}{2}, \qquad d=3, \\ 
	\frac{1+\sqrt{-7}}{2}, \qquad d=7, \\
	\frac{1+\sqrt{-11}}{2}, \qquad d=11.
	\end{cases}
\end{equation*}
Note that $\{ 1, \omega_d \}$ forms a $\Z$-basis of $\mathcal{O}_d$.

Presentations for our Bianchi groups are well-known; we will use the following ones (see \cite{Fine_1989}). We put 
\begin{equation*}
S = \begin{pmatrix}0 & -1 \\ 1 & 0\end{pmatrix},\ \ T = \begin{pmatrix}1 & 1 \\ 0 & 1\end{pmatrix}, \ \ T_\omega = \begin{pmatrix}1 & \omega_d \\ 0 & 1\end{pmatrix}.
\end{equation*}
Additionally, set $U=TS = \begin{psmallmatrix}
1 & -1 \\ 1 & 0
\end{psmallmatrix}$. In all cases, $S$ has order 2 and $U$ has order 3. For the cases $d=1,3$, where $\omega_d$ is a root of unity, we will also use the matrix
\begin{equation*}
	L = \begin{pmatrix} \omega_d & 0 \\ 0 & \omega_d^{-1} \end{pmatrix} .
\end{equation*}
The matrix $L$ has order 2 when $d=1$, and order $3$ when $d=3$. Throughout we will use matrices to refer to their images in $\PSL_2(\mathcal{O}_d)$.

\subsubsection{The case $d=1$}
\begin{align*}
\Gamma_1 \simeq \langle S,T,T_\omega,L \mid &  ~ S^2 = L^2 = (SL)^2 = (TL)^2 = (T_\omega L )^2= \\ & ~ (TS)^3=(T_\omega SL)^3 = [T,T_\omega] = 1 \rangle.
\end{align*}
Here $[g,h]=ghg^{-1}h^{-1}$ denotes the commutator. 

\begin{remark}
	Once in matrix form it can be easily verified that $L$ is not strictly necessary for the presentation, as it can be written $ST_\omega^{-1}ST_\omega S T_\omega^{-1}$. However, this substitution makes the presentation more complicated, so we leave it as-is. 
\end{remark}

\subsubsection{The case $d=2$}
\begin{equation*}
	\Gamma_2 \simeq \langle S,T,T_\omega \mid  S^2 = (TS)^3 = [T,T_\omega] = [S,T_\omega]^2 = 1 \rangle.
\end{equation*}

\subsubsection{The case $d=3$}\label{d3_presentation}
\begin{align*}
	\Gamma_3 \simeq \langle S, T, T_\omega, L \mid & ~ S^2 = (TS)^3 = (SL)^2 = (T^{-1} T_\omega S L)^3 = L^3 = \\ 
	& ~  L^{-1}T^{-1}T_\omega L T^{-1} = L^{-1} T L T_\omega = [T,T_\omega] = 1 \rangle 
\end{align*}
We note that this presentation is slightly different than the one appearing in \cite{Fine_1989}. This difference comes from the different basis for $\mathcal{O}_3$ used, and amounts to replacing occurrences of $U$ in their presentation with $T^{-1}T_\omega$. Our choice of basis of $\mathcal{O}_3$ comes from matching conventions in use on the LMFDB \cite{LMFDB}.

\begin{remark}
	As in the case $d=1$, the generator $L$ is not strictly necessary, as it can be written $STS T_\omega S T_\omega ^{-1} S T^{-1}$, but for the same reason we leave it as-is. 
\end{remark}

\subsubsection{The case $d=7$}
\begin{equation*}
\Gamma_7 \simeq \langle S,T,T_\omega \mid  S^2 = (TS)^3 = [T,T_\omega] = (T_\omega^{-1}ST_\omega S T)^2 = 1 \rangle.
\end{equation*}

\subsubsection{The case $d=11$}
\begin{equation*}
\Gamma_{11} \simeq \langle S,T,T_\omega \mid  S^2 = (TS)^3 = [T,T_\omega] = (T_\omega^{-1}ST_\omega S T)^3 = 1 \rangle.
\end{equation*}

\subsection{The weight modules}
Let $V_k(R)$ denote the space of polynomials in two variables $X$ and $Y$ over a ring $R$, homogeneous of degree $k$. The space $V_k(\C)$ is a right $\GL_2(\C)$-module by the action 
\begin{equation*}
	X^{n-i}Y^i \cdot \begin{pmatrix}
	a & b \\ 
	c & d
	\end{pmatrix} = (aX+bY)^{k-i}(cX+dY)^i.
\end{equation*}
We will also write $V_k$ when the ring is clear from context. We will consider the modules
\begin{equation*}
	V_{k,k} \coloneqq V_k \otimes \overline{V_k}.
\end{equation*}
We realise this module as complex polynomials in four variables, $X$, $Y$, $\Xbar$ and $\Ybar$, homogeneous in the pairs $X$, $Y$ and $\Xbar$, $\Ybar$ of degree $k$. 
The bar on the second component signifies that the action of $\GL_2(\C)$ there is twisted by complex conjugation, i.e. 
\begin{equation*}
\Xbar^{n-i}\Ybar^i \cdot \begin{pmatrix}
a & b \\ 
c & d
\end{pmatrix} = (\overline{a}\Xbar+\overline{b}\Ybar)^{n-i}(\overline{c}\Xbar+\overline{d}\Ybar)^i.
\end{equation*}

The space $V_{k,k}$ comes equipped with a natural pairing, arising from the pairings on its component spaces. On $V_{k}$ we have the pairing given by 
\begin{equation*}
	(  X^{k-\alpha}Y^{\alpha}, X^{k-\beta}Y^\beta  ) \mapsto (-1)^\alpha \binom{k}{\alpha}^{-1} \delta_{\alpha + \beta,k}.
\end{equation*}
The space $\overline{V_{k}}$ has the analogous pairing, and their product gives the pairing $\langle \cdot ,\cdot \rangle$ on $V_{k,k}$
\begin{equation*}
	\langle X^{k-\alpha}Y^\alpha \Xbar^{k-\gamma}\Ybar^{\gamma} , X^{k-\beta}Y^\beta \Xbar^{k-\varepsilon}\Ybar^{\varepsilon}  \rangle = (-1)^{\alpha + \gamma } \binom{k}{\alpha}^{-1} \binom{k}{\gamma}^{-1} \delta_{\alpha+\beta,k} \delta_{\gamma + \varepsilon,k}.
\end{equation*}
It can alternatively be described by the  formula 
\begin{equation} \label{pairing}
	\left \langle (aX+bY)^k (e\Xbar+f\Ybar)^k, (cX+dY)^k(g\Xbar+h\Ybar)^k \right \rangle = (ad-bc)^k(eh-fg)^k,
\end{equation}
from which follows the relation 
\begin{equation}\label{duality}
	\langle  P \cdot g,Q  \rangle  = \langle P,  Q\cdot g^\iota \rangle,
\end{equation}
where $g\in \GL_2(\C)$ and $g^\iota = \det(g)g^{-1}$. In particular, the dual of $g\in \SL_2(\C)$ is $g^{-1}$, so the pairing is $\Gamma_d$-invariant for all our $d$. 

In Section 5 we will consider the module $V_{k,k}(F)$ for number fields $F/K$. This is a vector space over $K$, which receives the action of $\GL_2(K)$ in the same manner described above (we identify the non-trivial Galois automorphism of $K$ with complex conjugation).

\subsection{Second cohomology} \label{second-cohomology}
For later use, we present a description of the second cohomology groups of our Bianchi groups. These arise from spectral sequences based on geometric data, see  \cite{Schwermer_1983} and \cite{Sengun_exp}.  
\subsubsection{The case $d=1$} Let $E = T_\omega S L = \begin{psmallmatrix}
-1 & \omega \\ 
\omega & 0
\end{psmallmatrix}$, which has order 3. For any $\Gamma_1$-module $M$, we have
\begin{equation} \label{H2-d1}
	H^2(\Gamma_1,M) \simeq M / (M^S + M^{SL} + M^U + M^E).
\end{equation}

\subsubsection{The case $d=2$} Let $A = T_{\omega}^{-1} S T_{\omega} S = \begin{psmallmatrix}
1 & \omega \\ 
\omega & -1
\end{psmallmatrix}$, which has order 2, and satisfies $ A = T_\omega \overline{A} T_\omega^{-1}$. For any $\Gamma_2$-module $M$, we have 
\begin{equation} \label{H2-d2}
	H^2(\Gamma_2,M) \simeq M/( M^U + M^S + M^A(1-T_{\omega}^{-1}))
\end{equation}

\subsubsection{The case $d=3$} In $\Gamma_3$, both $SL$ and $LS$ have order 2. For any $\Gamma_3$-module $M$, we have 
\begin{equation}
	H^2(\Gamma_3,M) \simeq M / (M^{LS} + M^U + M^{SL}).
\end{equation}

\subsubsection{The case $d=7$}\label{d7_H2} Let $A = ST_\omega^{-1}ST_\omega T^{-1}$, which has order 2, and $g = S T_\omega^{-1}$, which satisfies $g^{-1}Ag = \overline{A}$. For any $\Gamma_7$-module $M$, we have 
\begin{equation}
	H^2(\Gamma_7,M) \simeq M / (M^{S} + M^U + M^{A}(1+g)).
\end{equation}

\subsubsection{The case $d=11$}\label{d11_H2} Let $A= ST_\omega^{-1}ST_\omega T^{-1}$, which has order 3, and $g=ST_\omega^{-1}$, which satisfies $g^{-1}Ag = \overline{A}$. For any $\Gamma_{11}$-module $M$, we have 
\begin{equation}
H^2(\Gamma_{11},M) \simeq M / (M^{S} + M^U  + M^{A}(1+g)).
\end{equation}

\subsection{Modular symbols}
We define the space $\mathcal{M}_2$ of \demph{weight $2$ modular symbols} to be the quotient of the complex vector space with basis given by the symbols $\{ \alpha, \beta \}$, for $\alpha,\beta \in \mathbb{P}^1(K)=K \cup \{ \infty \}$, by the subspace spanned by the elements of the form $\{ \alpha, \beta \} + \{ \beta, \alpha \}$ and $\{\alpha, \beta \} + \{ \beta, \gamma \} + 
\{ \gamma, \alpha \}$. Put more compactly; 
\begin{equation*}
	\mathcal{M}_2 := \C [ \{ \alpha, \beta \} ] \bigg  / \ \bigg \langle \{ \alpha, \beta \} + \{ \beta, \alpha \} , \ \{ \alpha, \beta \} + \{ \beta, \gamma \} + \{ \gamma, \alpha \} \bigg \rangle.
\end{equation*}
The group  $\PSL_2(K)$ acts naturally on $\mathcal{M}_2$ via the formula
\begin{equation*}
	\{\alpha,\beta\}\cdot g = \{ g^{-1}\alpha, g^{-1}\beta \}
\end{equation*}
where $g^{-1} \alpha$ and $g^{-1} \beta$ are the usual linear fractional transformations.

Theorem 2.6  of \cite{Mohamed} tells us\footnote{In our case, his $r=1$ and his $\sigma_1$ equals our matrix $S$.} that the linear map 
\begin{equation} \label{Manin-surjection} \Phi^k =\Phi^k_d: V_{k,k} \longrightarrow \left ( \mathcal{M}_2 \otimes_\C V_{k,k} \right )_{\Gamma_d},
\end{equation}
given by
\begin{equation*}
	\Phi^k(v)= [ \{ 0, \infty \} \otimes v ]
\end{equation*}
is a {\em surjection}. Here the right hand side is the space of coinvariants
\begin{equation*}
	\left ( \mathcal{M}_2 \otimes_\C V_{k,k} \right )_{\Gamma_d} = \mathcal{M}_2 \otimes_{\C[\Gamma_d]} V_{k,k} = (\mathcal{M}_2 \otimes_\C V_{k,k}) / \langle u - ua \mid a \in \C[\Gamma_d] \rangle,
\end{equation*}
with $\Gamma_d$ and $\C[\Gamma_d]$ acting diagonally on $ \mathcal{M}_2 \otimes_\C V_{k,k}$. Moreover, $ [ \{ 0, \infty \} \otimes v ]$ denotes the class of $\{ 0, \infty\} \otimes v$. For convenience, we will call the above map $\Phi^k_d$ the {\bf Manin surjection} of weight $k$ for $\Gamma_d$. Note that this is not standard terminology.

\begin{remark}\label{non_euc}
	It is here that we use the restriction to the Euclidean fields. The Manin surjection is defined using all the non-parabolic generators of the associated Bianchi group. For a non-Euclidean imaginary quadratic field, there are more non-parabolic generators than just $S$. For example, for $K=\Q(\sqrt{-19})$, there are two non-parabolic generators, and so the Manin surjection has $V_{k,k} \oplus V_{k,k}$ as its domain, per \cite[Sections 2.1, 2.2]{Mohamed}. In principle, it should be possible to extend the methods we use to the non-Euclidean cases, however we leave this as a possibility for future work. 
\end{remark}

\subsection{Borel-Serre duality} 
The (complex) {\em Steinberg module} ${\rm St}$ of an arithmetic group $G$ is the $\C[G]$-module $H^{{\rm vcd}}(G,\C[G])$ where ${\rm vcd}$ is the virtual cohomological dimension of $G$. The Steinberg module is the dualising module of $G$ in the sense that it implements the duality theorem of Borel-Serre \cite[Thm. 11.4.3]{Borel_Serre_corners}: for any $\C[G]$-module $M$, we have an isomorphism 
\begin{equation}  H_j(G, {\rm St} \otimes_\C M) \simeq H^{{\rm vcd} - j}(G, M).
\end{equation}
This isomorphism respects the action of Hecke operators. 

For Bianchi groups $\Gamma$, the virtual cohomological dimension is 2, and we get the following corollary: for all $k$
\begin{equation} \label{borel-serre} H_0(\Gamma, {\rm St} \otimes_\C V_{k,k}) \simeq H^2(\Gamma, V_{k,k}).
\end{equation}
We will use this fact below.

The Steinberg module is well-known to be related to modular symbols; in our case of Bianchi groups, the Steinberg module is isomorphic to the space of weight 2 modular symbols that we defined above (see e.g. \cite{Ash_94}, \cite[Def. 4]{Torrey_12}). We record this below.
\begin{lemma} \label{steinberg=modular} For our Bianchi groups $\Gamma$, we have an isomorphism ${\rm St} \simeq \mathcal{M}_2$ of $\Gamma$-modules.
\end{lemma}
The details of this proof can be found in \cite[Section 4.1.1]{Mohamed}, although the definition of the Steinberg module is different to what we use here; the equivalence of the two definitions can be found in Section 1 of \cite{Ash_94}.

\subsection{The kernels} We will now argue that the kernel of the Manin surjection is exactly the subspace that was quotiented out in the previous explicit description of the second cohomology group in Section \ref{second-cohomology}. See also \cite[Prop. 1]{Torrey_12}.

We begin with a lemma. 
\begin{lemma} \label{fixed_in_kernel} Let $g \in \Gamma_d$ be an element of finite order, say $n$. Let $v \in V_{k,k}$ be fixed by $g$. Then $v$ lies in the kernel of $\Phi^k_d$.
\end{lemma}
\begin{proof} Let us put $A=1-g$ and $B=(1+g+\ldots+g^{n-1})$ for convenience. Observe that since $BA=0$ in the group algebra $\C[\Gamma_d]$, the kernel of $A$, as a linear map on $V_{k,k}$ equals the image of $B$. As $v$ is fixed by $g$, it lies in the kernel of $A$ and hence there is $v'$ such that $v=v'B$. We have 
\begin{equation*}
	\Phi(v)= [\{ 0, \infty \} \otimes v] =[ \{ 0, \infty \} \otimes v'B ] = [ \{ 0, \infty \}\cdot A \otimes v' BA ] =[0]
\end{equation*}
as claimed.
\end{proof}

\subsubsection{The cases $d=1,3$.} First, we do $d=1$, by showing that 
\begin{equation} \label{kernel-d1} \ker(\Phi^k_1) = V_{k,k}^S + V_{k,k}^{SL} + V_{k,k}^U + V_{k,k}^E,
\end{equation}
where the right hand side is as in the right hand side of (\ref{H2-d1}). 

Since the elements $S$, $SL$, $U$, and $E$ are all of finite order, Lemma \ref{fixed_in_kernel} implies that the right hand side of (\ref{kernel-d1}) is a subspace of the left hand side. To show equality, we will argue that the have the same dimension. We first employ Lemma \ref{steinberg=modular} to view $\Phi_1^k$ as a surjection onto 
\begin{equation*}
	H_0(\Gamma_1, {\rm St} \otimes_\C V_{k,k}) \simeq  H_0(\Gamma_1, \mathcal{M}_2 \otimes_\C V_{k,k}) \simeq \left ( \mathcal{M}_2 \otimes_\C V \right )_{\Gamma_1}.
\end{equation*}
Now, it follows from the isomorphisms (\ref{borel-serre}) and (\ref{H2-d1}) that the dimension of 
$ \ker(\Phi^k_1)$ has to be equal to that of its subspace $V_{k,k}^S + V_{k,k}^{SL} + V_{k,k}^U + V_{k,k}^E$. 
For $d=3$, exactly the same logic applied to $LS$, $U$ and $SL$ gives the result.

\subsubsection{The cases $d=2,7,11$} These cases are essentially the same, only also relying on the following lemma.
\begin{lemma} \label{fixed_space_conj}
	Let $A,B,g \in \Gamma$ such that $g^{-1} B g = A $. Then $V^A = V^Bg$.  
\end{lemma}
\begin{proof}
	First we show $V^A \subseteq V^B g$. Take $v\in V^A$. Then $v = v\cdot A = v\cdot g^{-1}Bg$. Multiplying by $g^{-1}$ on both sides, we get $v \cdot g^{-1}B = v\cdot g^{-1}$, i.e. $v\cdot g^{-1} \in V^B$, so $v \in V^B g$. The reverse inclusion follows the same argument, interchanging $A$ with $B$ and $g$ with $g^{-1}$. 
\end{proof}
For $d=2$, it is noted in Section \ref{H2-d2} that $ A = T_\omega \overline{A} T_\omega^{-1}$, and a quick check shows $A^2 = \overline{A}^2 = 1$, so $V_{k,k}^A(1-T_\omega^{-1}) = V_{k,k}^A + V_{k,k}^{\overline{A}}$ by Lemma \ref{fixed_space_conj}, and so the same argument as above shows 
\begin{equation*}
	\ker(\Phi^k_2) = V_{k,k}^S+ V_{k,k}^U + V_{k,k}^A(1-T_\omega^{-1}).
\end{equation*}
For $d=7,11$ the same argument works, again noting the relations between $A$ and $\overline{A}$ given in Sections \ref{d7_H2}, \ref{d11_H2}, giving 
\begin{equation*}
	\ker(\Phi^k_7) = V_{k,k}^S + V_{k,k}^U + V_{k,k}^A(1+ST_\omega^{-1})
\end{equation*} 
and
\begin{equation*}
	\ker(\Phi^k_{11}) = V_{k,k}^S + V_{k,k}^U + V_{k,k}^A(1+ST_\omega^{-1}).
\end{equation*}

\section{Period polynomials}
For our Euclidean Bianchi groups $\Gamma$ and weight modules $V_{k,k}$, let us consider the linear map 
\begin{equation} \label{eval_at_S_map} Z^1_{par}(\Gamma,V_{k,k}) \longrightarrow V_{k,k}, \qquad f \mapsto f(S)
\end{equation}
where the left hand side is the space of {\em parabolic} 1-cocycles; that is, those 1-cocycles $f : \Gamma \to V_{k,k}$ which vanish on the subgroup $\Gamma_\infty$ of elements in $\Gamma$ that stabilise the point $\infty \in \mathbb{P}^1(K)$ under fractional linear transformations. Observe that $\Gamma_\infty$ is simply the subgroup of upper triangular elements. The \demph{space $W_{k,k}$ of period polynomials} of weight $(k,k)$ for $\Gamma$ is defined as the subspace of $V_{k,k}$ given by the image of this map.

In \cite[Section 5]{Karabulut}, Karabulut gives explicit descriptions of the spaces of period polynomials for our Bianchi groups, which we recall.

\subsubsection{The case $d=1$}
In this case, we have 
\begin{equation} \label{d1_wkk}
	W_{k,k} = \ker(1+S) \cap \ker(1-L) \cap \ker(1+U+U^2) \cap \ker(1+E+E^2).
\end{equation}

\subsubsection{The case $d=2$}
In this case, we have 
\begin{equation}\label{d2_wkk}
	W_{k,k} = \ker(1+S) \cap \ker(1+U+U^2) \cap \ker(1+ST_\omega + T_\omega S + T_\omega^{-1} S T_\omega S).
\end{equation}

\subsubsection{The case $d=3$}In this case, we have 
\begin{equation}
	W_{k,k} = \ker(1+S) \cap \ker(1-L) \cap \ker(1+U+U^2) \cap \ker(1+E+E^2),
\end{equation}
where $E=T^{-1}T_\omega SL = \begin{psmallmatrix}
-1 & \omega -1 \\ 
\omega & 0
\end{psmallmatrix}$. As in Section \ref{d3_presentation}, we note that this definition is slightly different to that found in \cite{Karabulut}, again due to our choice of basis for $\mathcal{O}_3$. 

\subsubsection{The case $d=7$}In this case, we have 
\begin{equation}
	W_{k,k} = \ker(1+S) \cap \ker(1+U+U^2) \cap \ker(T + T_\omega S T + S T_\omega^{-1} S T_\omega+ ST_\omega).
\end{equation}

\subsubsection{The case $d=11$}In this case we have 
\begin{align*}
	W_{k,k} = & \ker(1+S) \cap \ker(1+U+U^2) \cap \\ & \ker(T + T_\omega S T + S T_\omega ^{-1} S T_\omega  + ST_\omega + T T_\omega^{-1}S T_\omega S T + ST_\omega T^{-1} S T_\omega^{-1} S T_\omega ).
\end{align*}

\subsection{Period polynomials and the Manin surjection} \label{period-Manin}
We now discuss the relationship between modular symbols and period polynomials. More precisely, we will prove that, for each of our Euclidean Bianchi groups, 
\begin{equation}
W_{k,k} = \left ( \ker(\Phi^k) \right ) ^\perp,
\end{equation}
where the right hand side denotes the subspace of all vectors in $V_{k,k}$ that are orthogonal under $\langle \cdot, \cdot \rangle$ to the kernel of Manin surjection $\Phi^k$ with respect to the pairing (\ref{pairing}). The proof in each case essentially boils down to checking certain relations (coming from the group algebra $\C[\Gamma_d]$) are satisfied. We will make copious use of the following basic fact from linear algebra:
\begin{lemma}\label{im_ker_swap}
	Let $F$ be an algebraically closed field of characteristic $0$, $V$ a finite-dimensional vector space over $F$, and $g$ an $F$-endomorphism of $V$ of finite order $n$. Then 
	\begin{equation*}
		\im(1-g) = \ker(1+g+\ldots + g^{n-1})
	\end{equation*}
	and 
	\begin{equation*}
		\ker(1-g) = \im(1 + g + \ldots + g^{n-1}). 
	\end{equation*}
\end{lemma}

\subsubsection{The case $d=1$}
 Recall from (\ref{kernel-d1}) that the kernel of $\Phi^k_1$ is equal to $V_{k,k}^S + V_{k,k}^{SL} + V_{k,k}^U + V_{k,k}^E$ which we will denote ${\bf V}$ for compactness. We rewrite 
\begin{equation*}
	{\bf V} = \ker(1-S) + \ker(1-SL) + \ker(1-U) + \ker(1-E) ,
\end{equation*}
and note that it follows from (\ref{duality}) that the kernel of $T$ is orthogonal to the image of $T^*$ for any element $T \in \C[\Gamma_1]$. Therefore, we have
\begin{equation*}
	{\bf V}^\perp = \im(1-S^{-1}) \cap \im(1-(SL)^{-1}) \cap \im(1-U^{-1}) \cap \im(1-E^{-1}),
\end{equation*}
and using Lemma \ref{im_ker_swap}, we get 
\begin{equation*}
	{\bf V}^\perp = \ker(1+S) \cap \ker(1 + LS) \cap \ker(1 + U + U^2) \cap \ker(1+ E + E^2).
\end{equation*}
Recalling the explicit description (\ref{d1_wkk}) of $W_{k,k}$ as
\begin{equation*}
	W_{k,k} = \ker(1+S) \cap \ker(1-L) \cap \ker(1+U+U^2) \cap \ker(1+E+E^2),
\end{equation*}
we see that, to prove $W_{k,k}={\bf V}^\perp$, we just need to prove that every $P \in W_{k,k}$ satisfies $P\cdot (1+LS)= 0$, and that every $Q \in {\bf V}^\perp$ satisfies $ Q\cdot (1-L) = 0$. 

Let $P\in W_{k,k}$, so $P=-P\cdot S$ and $P= P\cdot L $. Noting that $SLS=L$, we find 
\begin{align*}
	P\cdot (1+LS) & = P + P\cdot LS  \\ 
				  & = P - P\cdot SLS  \\ 
				  & = P\cdot (1-L) = 0.
\end{align*}
Hence, $W_{k,k} \subseteq {\bf V}^\perp$, and the same argument in reverse shows that ${\bf V}^\perp \subseteq W_{k,k}$.

\subsubsection{The case $d=2$}We proceed in the same way as the previous case. The kernel of $\Phi_2^k$ is ${\bf{V}} =   V_{k,k}^S + V_{k,k}^U +V_{k,k}^{A}(1-T_\omega^{-1}) $, where $A = T_{\omega}^{-1} S T_{\omega} S$. We want to use the same trick as before to compute the orthogonal complement, but in this case there is the extra factor of $1-T_\omega^{-1}$ on the last term. So we rewrite $\ker(1-A)$ as $\im(1+A)$, to get 
\begin{equation}\label{d2_tricky_kernel}
	{\bf V} = \ker(1-S) + \ker(1-U) +\im(1+A-T_\omega^{-1} - AT_\omega^{-1})   .
\end{equation}
Then, taking complements and turning images to kernels with Lemma \ref{im_ker_swap}, we get
\begin{equation*}
	{\bf V}^\perp = \ker(1+S) \cap \ker(1+U+U^2) \cap  \ker(1 + A - T_\omega - T_\omega A).
\end{equation*}
Again recalling the explicit description of $W_{k,k}$ in (\ref{d2_wkk}), we see that proving $W_{k,k} = {\bf V}^\perp$ amounts to showing that $P\in W_{k,k}$ satisfies $P\cdot (1-T_\omega + T_{\omega}^{-1} S T_{\omega}S - S T_\omega S)=0$, and that $Q\in {\bf V}^\perp$ satisfies $Q \cdot (1+ST_\omega  + T_\omega^{-1}S T_{\omega} S+ T_\omega S)=0$. 

Since $P\cdot (1+S) = 0$, we can expand to get 
\begin{align*}
	P\cdot (1-T_\omega + ST_\omega^{-1} S T_\omega - S T_\omega S) & = P - P \cdot T_\omega + P\cdot S T_\omega^{-1}S T_\omega - P\cdot ST_\omega S \\ 
	& = P + P \cdot S T_\omega + P\cdot S T_\omega^{-1}S T_\omega+  P\cdot T_\omega S \\
	& = P\cdot (1 + S T_\omega +S T_\omega^{-1}S T_\omega+T_\omega S) = 0,
\end{align*}
and, since $Q\cdot (1+S) =0$ also, the same logic works in reverse to give $W_{k,k} = {\bf V}^\perp$.

\subsubsection{The case $d=3$}
Proceeding as before, writing ${\bf V} = \ker(\Phi^k_3)$, we get 
\begin{equation*}
	{\bf V}^\perp = \ker(1+LS) \cap \ker(1+SL) \cap \ker(1+U+U^2).
\end{equation*}
First, we show $W_{k,k} \subseteq {\bf V}^\perp$. This amounts to showing $P\cdot (1+SL) = P\cdot (1+LS) = 0$ for all $P\in W_{k,k}$. 
To prove the first identity, we note that $P\cdot(1+S) = 0$, so $P\cdot (1+S)L=0$. Expanding and using that $P = P\cdot L$, we get $0 = P + P\cdot SL = P\cdot (1+SL)$. The second is proved in the same way, swapping the roles of $S$ and $L$. 

To show $ {\bf V}^\perp \subseteq W_{k,k}$, we must show 
\begin{align*}
	Q\cdot (1+S) = Q\cdot(1-L) = Q\cdot (1+E+E^2) = 0,
\end{align*}
for $Q\in {\bf V}^\perp$. If $Q$ \textit{is} in ${\bf V}^\perp$, it satisfies $Q\cdot SL = Q\cdot LS$, and since $LSL = S$ and $SL^2 = LS$, we have 
\begin{align*}
	Q\cdot (1+S) & = Q\cdot (1+LSL) \\ 
				 & = Q + (Q\cdot LS)\cdot L \\ 
				 & = Q + Q\cdot SL^2 = Q \cdot(1+LS) = 0.
\end{align*}
Using this, we also obtain 
\begin{align*}
	Q\cdot (1-L) & = Q - Q\cdot L \\
				 & = Q + Q \cdot SL = Q\cdot (1+SL) = 0.
\end{align*}
Finally, we show $Q\in \ker(1+E+E^2)$, where $E=T^{-1}T_\omega SL$. A quick computation shows $E$ can also be written $LTLS$, and $E^2 = LTSTLS$. Using $Q= Q\cdot L$, we get 
\begin{align*}
	Q\cdot(1+ E +E^2) & = Q + Q\cdot LTLS + Q\cdot LTSTLS \\ 
					  & = Q + Q\cdot TLS + Q\cdot (TS)TLS.
\end{align*}
Now we use that $Q\cdot TS = - Q - Q\cdot TSTS$ to get 
\begin{align*}
	Q \cdot (1+E+E^2) & = Q - Q\cdot TSTSTLS.
\end{align*}
Now, using that $(ST)^3=1$, as well as a final use of $Q = -Q\cdot S$, we get 
\begin{align*}
	Q \cdot (1+E+E^2) & = Q + Q\cdot STSTSTLS \\ 
					  & = Q + Q\cdot LS = Q\cdot (1+LS) = 0.
\end{align*}

\subsubsection{The case $d=7$} As before, write ${\bf V} = \ker(\Phi^k_7)$. We have 
\begin{equation*}
	{\bf V}^\perp = \ker(1+S) \cap \ker(1+U+U^2) \cap \ker(1 + T_\omega S + ST_\omega^{-1}ST_\omega T^{-1} + S T_\omega T^{-1}),
\end{equation*}
with the final kernel coming from $(V_{k,k}^A (1+g))^\perp = \ker((1+
g^{-1})(1+A))$, in the same manner as (\ref{d2_tricky_kernel}). Here $A = S T_\omega^{-1} S T_\omega T^{-1}$ and $g=ST_\omega^{-1}$. 

The only difference between ${\bf V}^\perp$ and $ W_{k,k}$ is the third kernel of each space. For $W_{k,k}$, this is $\ker(T+ T_\omega S T + S T_\omega^{-1} S T_\omega+ ST_\omega )$. But $T$ acts on $V_{k,k}$ injectively, so the two kernels are equal. 

\subsubsection{The case $d=11$} We have 
\begin{align*}
	{\bf V}^\perp = & \ker(1+S) \cap \ker(1+U+U^2) \cap \\ &  \ker(
	1 + T_\omega S + S T_\omega ^{-1} S T_\omega T^{-1} + ST_\omega T^{-1} + T T_\omega^{-1}S T_\omega S + ST_\omega T^{-1} S T_\omega^{-1} S T_\omega T^{-1}
	).
\end{align*}
Again, the final kernel comes from $(V_{k,k}^A(1+g)) ^\perp$. In this case it is equal to $\ker((1+g^{-1})(1+A+A^2))$, where $A= ST_\omega^{-1}ST_\omega T^{-1}$ and $g=ST_\omega^{-1}$. Comparing to the corresponding kernel for $W_{k,k}$, which is 
\begin{equation*}
	\ker(T + T_\omega S T + S T_\omega ^{-1} S T_\omega  + ST_\omega + T T_\omega^{-1}S T_\omega S T + ST_\omega T^{-1} S T_\omega^{-1} S T_\omega),
\end{equation*}
we again see the equality ${\bf V}^\perp = W_{k,k}$ from the injectivity of $T$.

\section{Hecke operators}
Recall that the space of modular symbols comes equipped with an action of Hecke operators; for a prime element $\pi \in \mathcal{O}$ (recall that $\mathcal{O}$ is Euclidean), we define an action
\begin{equation*}
\left (\{ \alpha, \beta \} \otimes v \right )\cdot T_\pi := \left (\{ \alpha, \beta \} \otimes v \right )\cdot \left( \left( \begin{smallmatrix} \pi &0 \\ 0 & 1 \end{smallmatrix}\right) + 
\sum_{\alpha \Mod \pi} \left (\{ \alpha, \beta \} \otimes v \right )\cdot \left ( \begin{smallmatrix} 1 & \alpha \\ 0 & \pi \end{smallmatrix}\right ) \right) .
\end{equation*}
for a weight $V_{k,k}$ modular symbol $\{ \alpha, \beta \} \otimes v \in \mathcal{M}_2 \otimes_\C V_{k,k}$. This action defines a linear map $T_\pi$ on the space $H_0(\Gamma, \mathcal{M}_2 \otimes_\C V_{k,k}) \simeq \left ( \mathcal{M}_2 \otimes_\C V_{k,k} \right )_{\Gamma}$.

\subsection{Heilbronn matrices}
Following \cite[Chapter 2]{Cremona-book} and \cite[Section 3.2]{Mohamed}, for a given prime element $\pi \in \mathcal{O}$, we define the associated set of {\bf Heilbronn matrices} as
\begin{equation*}
H_\pi = \left\{  \begin{pmatrix}
a & b \\ 
c & d
\end{pmatrix} \in M_2(\mathcal{O})~ \middle\vert ~  N(a) > N(b) \geq 0, N(d) > N(c) \geq 0, ad-bc = \pi   \right\}
\end{equation*}
where $N:K\to \Q$ is the norm map. Then, for $v \in V_{k,k}$, we define the operator $\widetilde{T}_\pi$ as  
\begin{equation*}
v \cdot \widetilde{T}_\pi := \sum_{g\in H_\pi} v\cdot g.
\end{equation*}

It is proven by Mohamed in \cite{Mohamed} that for any prime element $\pi \in \mathcal{O}$, we have a commutative diagram
\begin{equation*}
	\begin{tikzcd}
	{V_{k,k}} \arrow[r, "\Phi_d^k"] \arrow[d, "\tilde{T_{\pi}}"] & (\mathcal{M}_2\otimes V_{k,k})_{\Gamma_d} \arrow[d, "T_{\pi}"] \\
	{V_{k,k}} \arrow[r, "\Phi_d^k"]                              & (\mathcal{M}_2\otimes V_{k,k})_{\Gamma_d}                     
	\end{tikzcd}
\end{equation*}
where the horizontal arrows are the Manin surjection $\Phi^k_d$, see (\ref{Manin-surjection}). One can see this result as a transfer of the Hecke action from modular symbols to $V_{k,k}$. Alternatively we can say that when we equip $V_{k,k}$ with Heilbronn Hecke operators and modular symbols with the usual Hecke operators, the map $\Phi^k_d$ becomes {\em ``Hecke equivariant''}. 

Now, for each of the Euclidean $\Gamma_d$, we have computed the kernel of $\Phi^k_d$. The Hecke equivariance of $\Phi^k_d$ implies that these kernels are stabilised by the Heilbronn Hecke operators (since the trivial subspace of the right hand side is Hecke stabilised). Thus the action descends to the quotient and we obtain an isomorphism 
\begin{equation*}
	V_{k,k} / \ker(\Phi^k_d) \simeq \left ( \mathcal{M}_2 \otimes_\C V_{k,k} \right )_{\Gamma_d}
\end{equation*}
of Hecke modules.

\subsection{Heilbronn Hecke operators and period polynomials}
Recall that in Section \ref{period-Manin}, we showed that the space of period polynomials $W_{k,k}$ is the orthogonal complement of the kernel of the Manin surjection:
\begin{equation*}
	W_{k,k} = \ker(\Phi^k)^\perp,
\end{equation*}
with respect to the symmetric bilinear form (\ref{pairing}). Note that as (\ref{pairing}) is non-degenerate, we have
\begin{equation} \label{orthogonal-decoposition} V_{k,k} = W_{k,k} \oplus \ker(\Phi^k).
\end{equation}

Recall that for $g \in \GL_2(\C)$ and $v,w \in V_{k,k}$, we have
\begin{equation*}
	\langle v\cdot g, w\rangle = \langle v, w\cdot g^\iota\rangle
\end{equation*}
with $g^\iota = \det(g)g^{-1}$. Therefore, for a Heilbronn Hecke operator $\widetilde{T}_\pi$, we have
\begin{equation}\label{duality-2} 
	\langle v\cdot \widetilde{T}_\pi, w\rangle  = \langle v, w\cdot \widetilde{T}_\pi^*\rangle 
\end{equation}
where the operator $\widetilde{T}_\pi^*$ is given as
\begin{equation} \label{period-Heilbronn}
	w \cdot \widetilde{T}_\pi^* := \sum_{g\in H_\p} w\cdot g^\iota.
\end{equation}

Since the Heilbronn Hecke operators $\widetilde{T}_\pi$ stabilise the kernel of $\Phi^k$ and the space $W_{k,k}$ is orthogonal to this kernel, it follows from (\ref{duality-2}) that the {\em adjoint operators} $\widetilde{T}_\pi^*$ stabilise $W_{k,k}$. We record our discussion within the following corollary.
\begin{cor} For the Euclidean Bianchi groups, the space $W_{k,k}$ is equipped with an action of the adjoint Heilbronn Hecke operators as in (\ref{period-Heilbronn}). The canonical isomorphism
\begin{equation}
	\phi: W_{k,k} \simeq V_{k,k}/ \ker{\Phi^k_d}
\end{equation}
arising from the orthogonal decomposition (\ref{orthogonal-decoposition}) is Hecke equivariant, that is,
\begin{equation*}
	\phi(w\cdot \widetilde{T}_\pi^*) = \phi(w)\cdot \widetilde{T}_\pi
\end{equation*}
for any prime element $\pi \in \mathcal{O}_d$.
\end{cor}

\begin{remark}
We remark that all our discussion at this point could be adapted to the case of congruence subgroups of Bianchi groups. To compute with Bianchi forms of level $\Gamma_0(\n)$, one replaces the module $V_{k,k}$ with $V_{k,k} \otimes \mathbb{P}^1(\mathcal{O}/\n)$ (by an application of the Eckmann-Shapiro Lemma) in the above. This is done in detail in \cite{Mohamed}, and while we do not need it for our computations, we note it here for completeness.
\end{remark}

\section{Congruences}
In this section, we report on our numerical investigations into  the Hecke module structure of some spaces of Bianchi period polynomials. From the point of view of number theory, it is better to work with the group $\PGL_2(\mathcal{O})$, as only then the Hecke operators associated to the prime elements $\pi$ and $u\pi$ are the same for any unit $u \in \mathcal{O}^\times$, allowing one to associate Hecke operators to prime ideals. From the perspective of period polynomials, this amounts to computing the plus-subspaces that we now define.

\subsection{The plus-space}
Let $\varepsilon$ be a generator of the group $\mathcal{O}_d^\times$. The element $J=\begin{psmallmatrix} \varepsilon & 0 \\ 0 & 1 \end{psmallmatrix}$ normalises $\Gamma$ and thus gives rise to an involution\footnote{This involution is actually nothing but the Hecke operator associated to the double coset of $J$.} $\delta$ on cohomology and homology groups of $\Gamma$ which is well-known to commute with the Hecke operators. The involution $\delta$ on a 1-cocycle $f$ is given by 
\begin{equation*}
	\delta(f)(g)= f(JgJ^{-1})\cdot J.
\end{equation*}
When $\varepsilon=-1$, we have that $JSJ^{-1}=S$, and so the map (\ref{eval_at_S_map}) tells us that the corresponding involution on the space of Bianchi period polynomials is given by 
\begin{equation*}
	\delta(P)(X,Y,\overline{X},\overline{Y}) = P(X,Y,\overline{X},\overline{Y})\cdot J= P(-X,Y,-\overline{X},\overline{Y}).
\end{equation*}
When $\varepsilon \ne -1$ it is less obvious that $\delta$ is an involution, although it is---see the comment in Section 2.4 of \cite{Cremona_Whitley} for details. We define the \demph{plus-space} $W_{k,k}^+$ of Bianchi period polynomials to be the fixed subspace of the involution $\delta$ on $W_{k,k}$:
\begin{equation*}
	W_{k,k}^+ \coloneqq \{ P \in W_{k,k} \mid \delta(P)=P \}.
\end{equation*}
Notice that, as the involution $\delta$ on the space of Bianchi period polynomials commutes with the action of Heilbronn Hecke operators, the plus-space $W_{k,k}^+$ is stabilised under the Hecke action.

In \cite{Finis_etal}, the authors verified that, for the five Euclidean Bianchi groups, among all the spaces $S_{k+2}(\PGL_2(\mathcal{O}_d))$ of cuspidal Bianchi modular forms of full level $\PGL_2(\mathcal{O}_d)$ and weight $k+2$ within the scope given in the below table
\begin{center}
\begin{tabular}{|c|c|c|c|c|c|} \hline 
$-d$ & 1 & 2   &  3 & 7 & 11 \\  \hline
$k \leq$ & 104 & 141& 116 & 132 & 153  \\ \hline
\end{tabular}
\end{center}
there is a unique space which contains non-base-change (``genuine'') Bianchi modular forms; $S_{12}(\PGL_2(\mathcal{O}_{11}))$ is a 3-dimensional space, containing a pair $F_1,F_2$ of Galois conjugate genuine  forms, alongside the base-change lift of the classical discriminant modular form $\Delta$. We note that there is another space with non-zero genuine cuspidal dimension, for $d=-7$, $k=12$. However, these forms are in the ``minus space'' for the involution $\delta$, meaning they are forms of level $\PSL_2(\mathcal{O}_7)$, but not of level $\PGL_2(\mathcal{O}_7)$. This makes them less amenable to our purposes; in particular, due to their satisfying $\delta(P)=-P$, the first and last coefficients of their period polynomials are 0, preventing the use of the denominator method of Section \ref{BC_cong}. Further, their Hecke eigenvalues $a_\p$ are dependent on the choice of generator of the ideal $\p$ used, rather than the ideal itself, making them arithmetically less interesting. For these reasons, we make no further discussion of these forms.

In this section, using computer programs we developed to compute the spaces of Euclidean Bianchi period polynomials and the action of the associated Heilbronn Hecke operators, we exhibit two different congruences concerning the genuine Bianchi cuspforms living in $S_{12}(\PGL_2(\mathcal{O}_{11}))$; in the first, $F_1,F_2$ are congruent to (the base-change lift of) an Eisenstein series $E_{12}$ mod $173$, and in the second, they are congruent to $\Delta$ mod $43$. We summarise these congruences, along with a congruence between the base-change lifts of $\Delta$ and $E_{12}$, in a congruence graph (Figure \ref{FigCongGraph}), with vertices representing forms and edges representing moduli.
\begin{figure}[]\label{FigCongGraph}
	\centering
	\begin{tikzpicture}
	
	\coordinate[label=left:$A$]  (A) at (0,0);
	\coordinate[label=right:$B$] (B) at (4,0);
	\coordinate[label=above:$C$] (C) at (2,3.464);
	
	\draw [line width=0.5pt] (A) -- (B) -- (C) -- cycle;
	\node[circle,draw=black, fill=white, inner sep=0pt,minimum size=35pt] (b) at (0,0) {$\Delta$};
	\node[circle,draw=black, fill=white, inner sep=0pt,minimum size=35pt] (b) at (4,0) {$E_{12}$};
	\node[circle,draw=black, fill=white, inner sep=0pt,minimum size=35pt] (b) at (2,3.464) {$F_1,F_2$};
	
	\node[] (b) at (2,-0.5) {691};
	\node[] (b) at (0.5669,1.9820) {43};
	\node[] (b) at (3.4331,1.9820) {173};
	
	\end{tikzpicture}
	\caption{Congruence graph of Hecke eigenvalue systems captured in $S_{12}(\PGL_2(\mathcal{O}_{11}))$.}
\end{figure}

Included in the associated GitHub repo is a file, \texttt{check\_congruences.m}, which contains the Hecke eigenvalues for all primes of $K$ up to norm 1000, as well as code that checks they are congruent modulo the claimed primes. Using the method described in Section \ref{pol_congruences}, and carried out in Sections \ref{BC_cong}, \ref{genuine_cong}, \ref{cusp_cong}, this file also verifies that the period polynomials are congruent modulo the claimed primes.

These are, as far as we know, the first congruences in the literature concerning higher weight genuine Bianchi modular forms. A novelty concerning these congruences is that we ``detect'' them using Bianchi period polynomials, as we describe below. This is in analogy with the case of congruence subgroups of the classical modular group, as in, for example, \cite{Gaba_Popa}.

We will often refer to Hecke eigenvectors in the space of Bianchi period polynomials as Bianchi modular forms. This is justified under the embeddings discussed in the introduction and our results concerning the Hecke action.

\subsection{Cusp forms and algebraicity} Let us start the discussion by summarising some algebraicity results for our Bianchi period polynomials which follow from well-known results of Hida, see \cite[Sections 3, 6 and 8]{Hida_94}. 

Let $f \in S_{k+2}(\PGL_2(\mathcal{O}))$ be a level 1 Bianchi cusp form of weight $k+2$. As proven by Harder, and explicated by Hida, there is a $V_{k,k}(\C)$-valued harmonic differential 1-form $\omega_f$ on the arithmetic 3-fold $X_\Gamma$ (given by the orbit space of $\Gamma$ on the hyperbolic $3$-space $\mathcal{H}=\C \times \R^+$). The Eichler-Shimura-Harder map $\Theta$ mentioned in the introduction takes $f$ to the class of the 1-cocycle
\begin{equation*}
	\Theta(f): \Gamma \to V_{k,k}(\C), \qquad \gamma \mapsto \int_z^{\gamma\cdot z} \omega_f
\end{equation*}
by integrating the form $\omega_f$ over the geodesic from $z$ to $\gamma\cdot z$. Here $z$ is some fixed choice of point in $\mathcal{H}$ or in the rational boundary $\mathbb{P}^1(K)=K \cup \{ \infty \}$ of $\mathcal{H}$. Changing the point $z$ changes the cocycle by a coboundary. Note that, if we choose $z=\infty$, the cocycle we obtain is parabolic. 

Using the ``evaluation at $S$'' map of the introduction, we can obtain a period polynomial from a parabolic cocycle. The {\bf canonical\footnote{This polynomial is canonical because $\omega_f$ is canonical once $f$ is appropriately scaled.} period polynomial} associated to $f \in S_{k+2}(\PGL_2(\mathcal{O}))$ is defined as 
\begin{equation} 
	r_f := \int_0^\infty \omega_f \in W_{k,k}^+(\C).
\end{equation}
\begin{remark}\label{L_coeffs}
	As in the case of classical cusp forms, the coefficients of the canonical period polynomial of $f$ are related to the special values of the $L$-function of $f$: the critical values $L(f,s+1)$ with $0 \leq s \leq k$ appear in the coefficients of the diagonal terms $X^{k-s}Y^{s}\overline{X}^{k-s}\overline{Y}^{s}$, whereas off-diagonal terms are related to values of the twisted $L$-functions of $f$, see \cite[Thm. 2.11]{Williams_17} for details. 
\end{remark}
Now assume that $f$ is a simultaneous eigenvector for all the Hecke operators. It is a classical result that the field extension obtained by adjoining all the Hecke eigenvalues of $f$ to $K$ is a number field which we will denote by $K(f)$. It follows from results of Hida that there is a finite extension $F/K(f)$ and a complex period $\Omega_f \in \C^\times$ such that 
\begin{equation*}
	\tfrac{1}{\Omega_f} r_f \in W_{k,k}(F)^+.
\end{equation*}
Our computer programs compute the space $W_{k,k}(K)^+$, which is stabilised by Hecke operators. Since we have 
\begin{equation*}
	W_{k,k}^+(K) \otimes_K \C \simeq W_{k,k}^+(\C),
\end{equation*}
our Hecke operators capture all of the Hecke module information of $W_{k,k}^+(\C)$, and hence of $S_{k+2}(\PGL_2(\mathcal{O}))$. By suitably replacing the coefficient field $K$ with a finite extension (coming from the characteristic polynomials of the Hecke operators) we obtain {\bf algebraic} period polynomials (spanning a $K$-line) which realise the Hecke eigenvalue systems that we observe. 

\subsection{Eisenstein series}
It is important to note that the space $S_{k+2}(\PGL_2(\mathcal{O}))$ does not account for all of $W_{k,k}^+(\C)$; there is a one-dimensional complement which ``sees'' an Eisenstein series. Here is one way to explain this.

We have established that there is a Hecke equivariant isomorphism
\begin{equation*}
	W_{k,k}^+(\C) \simeq H^2(\Gamma,V_{k,k}(\C));
\end{equation*}
in turn, it is well-known that there is also a Hecke equivariant isomorphism  
\begin{equation*}
	H^2(\Gamma,V_{k,k}(\C) \simeq S_{k+2}(\Gamma) \oplus \mathrm{Eis}_{k+2}(\Gamma),
\end{equation*}
where, for $k>0$, the latter space is the one-dimensional\footnote{For $k=0$, the latter space is trivial, see \cite[Prop. 1]{Rahm_Sengun}.} space generated by the Eisenstein series $E_{k+2}$ of weight $k+2$ associated to the single cusp of the arithmetic hyperbolic 3-fold $X_\Gamma$. 

We therefore see that there is a one-dimensional complement to the image of $S_{k+2}(\Gamma)$ in $W_{k,k}^+(\C)$, coming from $E_{k+2}$ and realising the eigenvalue system $\{ N(\p)^{k+1}+1 \}$. One can show that this one-dimensional complement is spanned by the period polynomial $X^k\overline{X}^k - Y^k\overline{Y}^k$, corresponding to the image of the one-dimensional space of parabolic coboundaries $B^1_{par}(\Gamma,V_{k,k}(\C)) \subset Z^1_{par}(\Gamma, V_{k,k}(\C)) \to W_{k,k}(\C)$.

This is the Bianchi counterpart of the classical period polynomial $X^k-Y^k$, which is well-known (going back to \cite{Manin_1973}, see also \cite[Thm. 1]{Kohnen_Zagier_1984}) to realise the Eisenstein eigenvalue system $\{ p^{k+1}+1 \}$. We will not prove as we do not need it; however we will see that it holds in the specific Bianchi period polynomial space that we will compute with in the rest of the paper.

\subsection{Congruences of period polynomials}\label{pol_congruences}
Two algebraic period polynomials $r_f$ and $r_g$ each have their own field of definition, $K(f)$ and $K(g)$ respectively. Both polynomials can be defined over the compositum of these fields, $K(f,g) \coloneqq K(f)K(g)$, and can therefore be scaled to have coprime\footnote{Meaning the ideals generated by the coefficients are coprime.} coefficients in the ring of integers $\mathcal{O}_{f,g}$ of $K(f,g)$. We write $R_f$ and $R_g$ for the integrally scaled polynomials. For any prime ideal $\p$ of $\mathcal{O}_{f,g}$, there is a mod $\p$ reduction map 
\begin{equation*}
	W_{k,k}(\mathcal{O}_{f,g}) \to W_{k,k}(\F_\p)
\end{equation*}
induced by the map reduction mod $\p$ map $\mathcal{O}_{f,g} \to \F_\p \simeq \mathcal{O}_{f,g}/\p$. The polynomials $R_f$ and $R_g$ are \demph{congruent mod $\p$} if their images in $W_{k,k}(\F_\p)$ are equal. When two polynomials are congruent and neither of their images is the zero polynomial in $W_{k,k}(\F_\p)$, their Hecke eigenvalues are automatically congruent modulo $\p$, as the reduction map is Hecke equivariant. We will use this fact to prove the existence of congruences between Bianchi modular forms of level 1 in the rest of this section.

\subsection{The space}
For the remainder of this section, we set $K=\Q(\sqrt{-11})$ and $\Gamma=\PGL_2(\mathcal{O}_{11})$. As expected, we compute that the space 
\begin{equation*}
	W\coloneqq W_{10,10}^+(K) \otimes_K \C \simeq W_{10,10}^+(\C)
\end{equation*}
is 4-dimensional, consisting of the base-change of the classical cusp form $\Delta$, which we still denote by $\Delta$, the Eisenstein series $E_{12}$, which is the base-change of the classical level 1 weight 12 Eisenstein series, and two genuine Bianchi cusp forms $F_1,F_2$. The genuine forms have Hecke eigenvalues in the field $L = \Q(\beta)$, where $\beta = \sqrt{81829}$, and are conjugate by the non-trivial automorphism of $L$. We note that the primes that ramify in $L$ are exactly $\{ 11,43,173 \}$. We record the Hecke eigenvalues of the forms for the first few primes of $K$ in Figure \ref{eigenvalues}. Note that the Hecke eigenvalues for the genuine forms agree with those computed in \cite[Section 6.2.2]{Finis_etal}.

\begin{figure}[]
	\centering
	\begin{table}[H]\renewcommand{\arraystretch}{1.1}
		\begin{tabular}{c||c|c|c|c|c}
			$\p$   & $\omega$ & $1-\omega$ & $2$ & $1+\omega$ & $2-\omega$   \\ \hline
			$N(\p)$ & 3 & 3 & 4 & 5 & 5 \\ \hline \hline
			$a_\p(\Delta)$ &  252   &  252     & -3250  &  4830     &  4830           \\ \hline 
			$a_\p(E_{12})$   &  177148   &    177148   & 4194305  &   48828126    &   48828126   \\ \hline 
			$a_\p(F_1)$& $\beta-350$ &     $-\beta - 350$ & $-80$ & $26\beta - 5103$ & $-26\beta - 5103$  \\ \hline 
			$a_\p(F_2)$& $-\beta - 350$ &  $\beta-350$    & $-80$ & $-26\beta - 5103$  & $26\beta - 5103$ 
		\end{tabular}\renewcommand{\arraystretch}{1}
	\end{table}
	\caption{Hecke eigenvalue systems captured in $S_{12}(\PGL_2(\mathcal{O}_{11}))$.}\label{eigenvalues}
\end{figure}

\subsection{Congruence between $\Delta$ and $E_{12}$}\label{BC_cong}
It immediately follows from the recipe\footnote{Which can be found, e.g. here \href{https://www.lmfdb.org/knowledge/show/mf.bianchi.base_change}{\texttt{https://www.lmfdb.org/knowledge/show/mf.bianchi.base\_change}}} describing the behaviour of Hecke eigenvalues under base-change lifting that the famous congruence mod $691$ between the classical modular forms $\Delta$ and $E_{12}$ continues to hold between their base-change lifts.

The eigenvalue system associated to $E_{12}$ cuts out a one-dimensional subspace of $W_{10,10}^+(K)$ which we see to be spanned by $X^{10}\overline{X}^{10} - Y^{10}\overline{Y}^{10}$.

Now consider the one-dimensional subspace of $W_{10,10}^+(K)$ cut out by (the eigenvalue system associated to) $\Delta$. We pick any non-zero Bianchi period polynomial in this one-dimensional subspace, and scale it\footnote{In fact, \textsf{Magma} by default gives us an eigenvector whose leading coefficient is $1$.} so that the first coefficient (i.e. the coefficient of the monomial $X^{10}\overline{X}^{10}$) is $1$.  
Let us call this polynomial $P$ and write 
\begin{equation*}
	P= \sum_{0 \leq i,j \leq k} c_{i,j} X^{10-i}Y^i\Xbar^{10-j}\Ybar^j.
\end{equation*}
We represent $P$ as an $11\times 11$ coefficient matrix $(c_{i,j})_{0 \leq i,j \leq 10}$:

\begin{figure}[H]
	\centering
$\begin{psmallmatrix}
1 & 0 & \frac{55280}{31203} & 0 & \frac{5528}{10401} & 0 & \frac{3455}{55472} & 0 & \frac{691}{218421} & 0 & 0\\
0 & -\frac{7082750}{1965789} & 0 & -\frac{193480}{93609} & 0 & -\frac{34550}{93609} & 0 & -\frac{65645}{2246616} & 0 & 0 & 0\\
\frac{55280}{31203} & 0 & \frac{4892971}{1497744} & 0 & \frac{79465}{93609} & 0 & \frac{8983}{93609} & 0 & 0 & 0 & -\frac{691}{218421}\\
0 & -\frac{193480}{93609} & 0 & -\frac{21044405}{15726312} & 0 & -\frac{13820}{93609} & 0 & 0 & 0 & \frac{65645}{2246616} & 0\\
\frac{5528}{10401} & 0 & \frac{79465}{93609} & 0 & \frac{156857}{499248} & 0 & 0 & 0 & -\frac{8983}{93609} & 0 & -\frac{3455}{55472}\\
0 & -\frac{34550}{93609} & 0 & -\frac{13820}{93609} & 0 & 0 & 0 & \frac{13820}{93609} & 0 & \frac{34550}{93609} & 0\\
\frac{3455}{55472} & 0 & \frac{8983}{93609} & 0 & 0 & 0 & -\frac{156857}{499248} & 0 & -\frac{79465}{93609} & 0 & -\frac{5528}{10401}\\
0 & -\frac{65645}{2246616} & 0 & 0 & 0 & \frac{13820}{93609} & 0 & \frac{21044405}{15726312} & 0 & \frac{193480}{93609} & 0\\
\frac{691}{218421} & 0 & 0 & 0 & -\frac{8983}{93609} & 0 & -\frac{79465}{93609} & 0 & -\frac{4892971}{1497744} & 0 & -\frac{55280}{31203}\\
0 & 0 & 0 & \frac{65645}{2246616} & 0 & \frac{34550}{93609} & 0 & \frac{193480}{93609} & 0 & \frac{7082750}{1965789} & 0\\
0 & 0 & -\frac{691}{218421} & 0 & -\frac{3455}{55472} & 0 & -\frac{5528}{10401} & 0 & -\frac{55280}{31203} & 0 & -1\\
\end{psmallmatrix}$
\end{figure}

Guided by our desire to find a congruence with $E_{12}$, we let $D$ denote the gcd of all the rational numbers that appear as ``middle coefficients'' (i.e. all the coefficients except for the first $(i,j)=(0,0)$ and the last $(i,i)=(10,10)$). That is, 
\begin{equation*}
	D= \dfrac{\gcd \big (\{ {\rm numerator}(c_{i,j}) \mid (i,j)\neq (0,0),(10,10)\} \big )}{\mathrm{lcm} \big (\{ {\rm denominator}(c_{i,j}) \mid (i,j)\neq (0,0),(10,10) \} \big )} = \frac{691}{31452624},
\end{equation*}
so that when we scale $P$ by $1/D$, the middle coefficients all become integers with greatest common divisor 1, giving the scaled period polynomial 
\begin{figure}[H]
	\centering
	$\begin{psmallmatrix}
		\frac{31452624}{691} & 0 & 80640 & 0 & 24192 & 0 & 2835 & 0 & 144 & 0 & 0 \\
		0 & -164000 & 0 & -94080 & 0 & -16800 & 0 & -1330 & 0 & 0 & 0 \\
		80640 & 0 & 148701 & 0 & 38640 & 0 & 4368 & 0 & 0 & 0 & -144 \\
		0 & -94080 & 0 & -60910 & 0 & -6720 & 0 & 0 & 0 & 1330 & 0 \\
		24192 & 0 & 38640 & 0 & 14301 & 0 & 0 & 0 & -4368 & 0 & -2835 \\
		0 & -16800 & 0 & -6720 & 0 & 0 & 0 & 6720 & 0 & 16800 & 0 \\
		2835 & 0 & 4368 & 0 & 0 & 0 & -14301 & 0 & -38640 & 0 & -24192 \\
		0 & -1330 & 0 & 0 & 0 & 6720 & 0 & 60910 & 0 & 94080 & 0 \\
		144 & 0 & 0 & 0 & -4368 & 0 & -38640 & 0 & -148701 & 0 & -80640 \\
		0 & 0 & 0 & 1330 & 0 & 16800 & 0 & 94080 & 0 & 164000 & 0 \\
		0 & 0 & -144 & 0 & -2835 & 0 & -24192 & 0 & -80640 & 0 & -\frac{31452624}{691} 
	\end{psmallmatrix}$
	\caption{The ``normalised'' algebraic period polynomial of $\Delta$ over $\Q(\sqrt{-11})$.}
\end{figure}
In analogy with the classical case (see \cite[Section 7.1]{Manin_1973}), we expect the denominator of the first (or last) coefficient to give the modulus of a congruence. In this case we see a congruence modulo 691, and a 691 in the denominator.

This scaling gives the normalised algebraic period polynomial $r_\Delta$. To detect a congruence with the Eisenstein series $E_{12}$, we need the \textit{integrally} scaled polynomials $R_\Delta$ and $R_{E_{12}}$. We obviously have $R_{E_{12}} = r_{E_{12}} = X^{10}\overline{X}^{10} - Y^{10}\overline{Y}^{10}$, and $R_\Delta = 691r_\Delta$. 

From the above representation of $r_\Delta$, it is clear every coefficient of $R_\Delta$ except the first and last is divisible by 691. So the reduction of $R_\Delta$ modulo 691 is
\begin{equation*}
	R_\Delta \equiv 377X^{10}\overline{X}^{10} - 377Y^{10}\overline{Y}^{10} \equiv 377 R_{E_{12}} \Mod{691}.
\end{equation*}
Thus (although we already knew this), we have 
\begin{equation*}
	\Delta \equiv E_{12} \Mod{691}.
\end{equation*}

\begin{remark} \label{symmetry}
	The representation of our Bianchi period polynomials as square matrices has no particular significance; the polynomial is a vector in the space $W_{10,10}(K)^+$ and has no reasonable interpretation as a matrix that we know of. It is an aesthetic decision, as this representation fits more easily on the page, and its symmetries reflect some obvious symmetries of the polynomial under the action of $\PSL_2(\mathcal{O})$. For example, the matrix $S = \begin{psmallmatrix} 0  & -1 \\ 1 & 0\end{psmallmatrix}$ acts on polynomials by $P(X,Y,\Xbar,\Ybar) \cdot S = P(-Y,X,-\Ybar,\Xbar)$, which flips the matrix along the horizontal and vertical lines of symmetry and negates entries $(i,j)$ with $i+j$ odd. Since $r_\Delta \cdot (1-J) =0$, where $J = \begin{psmallmatrix} -1 & 0 \\ 0 & 1\end{psmallmatrix}$, all such entries are 0. 
\end{remark}

\begin{remark} \label{only_diagonal} It may be worth noting that it appears to be possible to detect the prime $691$ using only the ``diagonal'' coefficients of the period polynomial, i.e. those corresponding to the monomials $X^{10-i}Y^i\Xbar^{10-i}\Ybar^i$, for $0\leq i \leq 10$. Per Remark \ref{L_coeffs}, these correspond to critical values of the $L$-function of $\Delta$, and come from \textsf{Magma} unscaled as 
\begin{equation*}
	\left(1,	-\tfrac{7082750}{1965789},\tfrac{4892971}{1497744},-\tfrac{21044405}{15726312},
	\tfrac{156857}{499248},0,-\tfrac{156857}{499248},	\tfrac{21044405}{15726312},-\tfrac{4892971}{1497744},	\tfrac{7082750}{1965789},-1\right).
\end{equation*}
Taking gcds as before reproduces the previous scaling, again putting a 691 in the denominator of the first and last coefficients.
\end{remark}

\subsection{Congruence between the genuine cusp forms and $E_{12}$}\label{genuine_cong}
In the same manner as the base-change congruence, our goal is to compute the normalised period polynomial $r_{F_1}$ and use it to detect a congruence with $E_{12}$. As the Hecke eigenvalues of $F_1$ live in the quadratic field $L = \Q(\sqrt{81829})$, any algebraic eigen-polynomial $P$ realising the Hecke eigenvalue system of $F_1$ will live in $W_{10,10}(L')^+$, where $L'$ is the compositum $LK$, a field of degree 4. We compute one such polynomial $P$ and as before write
\begin{equation*}
P= \sum_{0\leq i,j \leq 10}c_{i,j}X^{10-i}Y^i\Xbar^{10-j}\Ybar^{j}.
\end{equation*}
Again, we scale $P$ so that its leading coefficient is $1$. In our next step, we would like to take the gcd of all the middle coefficients, however, since these coefficients lie in the number field $L'$ whose class number\footnote{The field $L'$ has class number 116.} is not 1, we have to consider the gcd\footnote{Here the notion of gcd for fractional ideals is the straightforward generalisation of the same for rational numbers used in Section \ref{BC_cong}.} of the fractional ideals generated by the middle coefficients:
\begin{equation*}
	\mathcal{D} = \dfrac{\gcd \big (\{ \langle {\rm numerator}(c_{i,j}) \rangle \mid (i,j)\neq (0,0),(10,10)\} \big )}{\mathrm{lcm} \big (\{ \langle {\rm denominator}(c_{i,j}) \rangle \mid (i,j)\neq (0,0),(10,10) \} \big )}.
\end{equation*}

We find that the fractional ideal $\mathcal{D}$ is principal, generated by the element
\begin{equation*}
	D=(245047560419778865\cdot T^3 + 491449950388685970467\cdot T)/15099638400\in L',
\end{equation*}
where $T$ is a root of the polynomial $x^4 + 3725 x^2 + 3448449$ so that $\Q(T) \simeq L'$. Now multiplying our period polynomial $P$ by $1/D$, we obtain the normalised algebraic period polynomial $r_{F_1}$, whose middle coefficients all lie in the ring of integers of $L'$ and generate ideals that are all pairwise coprime. 

The key observation now is that the leading coefficient of our final polynomial is $1/D$, which has norm 
\begin{equation*}
	N_{L'/\Q}(1/D) = \dfrac{2^{28} \cdot 3^4 \cdot 5^8 \cdot 7^4 \cdot 11^8}{173^2}.
\end{equation*}

This indicates the existence of a congruence modulo 173. To detect such a congruence at the level of period polynomials, we need to scale $r_{F_1}$ such that \textit{all} its terms are integral and coprime. A quick calculation shows that 
\begin{equation*}
	D = \alpha/8131200,
\end{equation*}
where $\langle \alpha \rangle$ is the unique prime of $L'$ of norm $173^2$, and so $R_{F_1} = \alpha r_{F_1}$. As we saw for the previous congruence, all the middle coefficients of $R_{F_1}$ are divisible by $\alpha$, so modulo $\langle \alpha \rangle$, $R_{F_1}$ is in the space generated by $R_{E_{12}}$. Writing $\p_{173}$ for the unique prime of norm 173 in $L$, another quick calculation shows that $\langle \alpha \rangle \cap \mathcal{O}_L = \p_{173}$, i.e. $\p_{173}$ is inert in $L'/L$. In particular, we have 
\begin{equation*}
	R_{F_1} \equiv 27 R_{E_{12}} \Mod{\langle \alpha \rangle }.
\end{equation*}
From this we conclude that 
\begin{equation*}
	F_1 \equiv E_{12} \Mod{\p_{173}}.
\end{equation*}

\begin{remark} As in Remark \ref{only_diagonal}, we also examine the unscaled diagonal entries of $F_1$:
\begin{equation*}
	    \left(1,-\frac{3287}{924},\frac{42039}{13552},-\frac{173}{154},	\frac{173}{968},0,-\frac{173}{968},\frac{173}{154},-\frac{42039}{13552},\frac{3287}{924},-1\right)
\end{equation*}
Surprisingly, all the diagonal coefficients are rational, and scaling them as before, we get 
\begin{equation*}
	\left( 
	\frac{40656}{173},-836,	729,-264,42,0,-42,264,-729,	836,-\frac{40656}{173}
	\right).
\end{equation*}
Although this choice of scaling does \textit{not} reproduce the scaling obtained from considering the whole polynomial, we do see the factor of 173 reappearing the denominator of the first and last terms. We don't know if scaling via the diagonal coefficients should always be sufficient to detect a congruence, or if one can expect there to be a choice of scaling that makes the diagonal rational in all cases. 
\end{remark}

The interested reader can find all the unnormalised algebraic period polynomials (including those corresponding to $F_1$ and $F_2$, which we cannot display here due to their size) in the associated GitHub repository, in the file \texttt{Q11\_periodPols.m}.

\subsection{Congruence between the genuine cusp forms and $\Delta$}\label{cusp_cong}
Finally, we look to detect congruences between the genuine cusp forms $F_1,F_2$ and the lifted form $\Delta$. It is a well-known result of Hida that such congruences are controlled by the size of a certain ``congruence module'' which, by a result of Urban \cite{Urban_95} for Bianchi cusp forms, is captured by the algebraic part of the value of the adjoint $L$-function at $s=1$ (which is obtained by factoring out suitable pair of complex periods).

Further, it is well-known that the adjoint $L$-function at $s=1$ is closely related to the Petersson norm. In turn, a result of Haberland \cite{Haberland}, see also  \cite{Kohnen_Zagier_1984, Pasol-Popa}, tells us that one can compute the Petersson norm of a classical cusp form via its canonical period polynomial. This motivates us to try to detect congruences between our genuine cusp forms and $\Delta$ using our Bianchi period polynomials.

Let $f \in S_{k+2}(\Gamma)$ be a Bianchi cusp form and let $\omega_f$ be its $V_{k,k}(\C)$-valued harmonic differential form on the associated hyperbolic 3-fold $X_\Gamma$. Recall that the Petersson norm of $f$ is given by
\begin{equation*}
	(f,f)=\int_{X_\Gamma} \omega_f \wedge {}^*\omega_f
\end{equation*}
where ${}^*\omega_f$ is the harmonic $2$-form given by the Hodge-star of $\omega_f$. 

We believe that a suitable analogue of Haberland's formula which expresses the Petersson norm of $f$ in terms of its canonical period polynomials ought to exist, however our preliminary efforts failed to establish such a formula. Instead, we have developed a computational approach that we believe captures the algebraic part of the Petersson norm. We explain this approach now. 

Any Bianchi cusp eigenform $f \in S_{k+2}(\Gamma)$ gives rise to a class in $H^1(\Gamma,V_{k,k}(\C))$ as we have discussed earlier, but also to a class\footnote{The description of the $V_{k,k}(\C)$-valued harmonic differential 2-form associated to $f$ can be found in \cite[Sec. 3]{Hida_94}.} in $ H^2(\Gamma,V_{k,k}(\C))$. As previously discussed, an algebraic period polynomial can be associated to the class of $f$ in $H^1$. For the second cohomology class, we use the fact that there is a Hecke-module isomorphism 
\begin{equation*}
	H^2(\Gamma,V_{k,k}(\C)) \simeq V_{k,k}(\C) / {\bf V}(\C).
\end{equation*}
(see Section \ref{second-cohomology}). Using computer programs that we developed, we can compute an algebraic vector $v_f$ in $V_{k,k}(K(f))$ such that the class of $v_f$ in $V_{k,k}(K(f)) / {\bf V}(K(f))$ realises the Hecke eigenvalue system of $f$. Thus, the class of $v_f$ corresponds to the cohomology class associated to $f$ in $H^2$. We then take the pairing $\langle r_f, v_f \rangle$. We note that the value of $\langle r_f, v_f \rangle $ does not change if we replace $v_F$ with $v_F+{\bf v}$ for any ${\bf v} \in {\bf V}$, since ${\bf V}$ and $W_{k,k}$ are orthogonal under $\langle \cdot,\cdot\rangle$. Our expectation is that the algebraic quantity $\langle r_f, v_f \rangle $ captures the algebraic part of the Petersson norm of $f$.

We now apply this strategy to the forms in $W_{10,10}(\PGL(\mathcal{O}_{11}))$, looking first at $\Delta$. We have already written down a normalised algebraic period polynomial corresponding to $\Delta$ in Section \ref{BC_cong}. We also compute an algebraic vector $v_\Delta$ in $V_{10,10}(K)$ whose class in $V_{10,10}(K) / {\bf V}(K)$ realises the Hecke eigenvalue system of $\Delta$:
\begin{equation*}
v_\Delta = -358X^{10}\Xbar^8\Ybar^2 + 3080X^{10}\Xbar^6\Ybar^4 + 22253X^{10}\Xbar^4\Ybar^6.
\end{equation*}
It is not clear exactly how to normalise $v_\Delta$. For lack of a better convention, we have stuck to eliminating denominators and making the gcd of all (except the first and last) coefficients $1$. We then compute 
\begin{equation}\label{pairing_Delta_Delta}
	\langle r_{\Delta},v_\Delta \rangle = \dfrac{7^2 \cdot 13 \cdot 43} {2}.
\end{equation}
Since $7$ is less than the weight $10$, we do not expect to see a congruence at that prime. There is no cusp form in $W_{10,10}$ congruent to $\Delta$ modulo 13, but we \textit{do} observe that
\begin{equation*}
	\Delta \equiv F_1 \Mod{\p_{43}},
\end{equation*}
for the first few primes $\p$ of $K$, with $\p_{43}$ the unique prime of $L$ of norm $43$. To show this congruence holds for all primes of $K$, we need to check the integrally scaled period polynomials $R_\Delta$ and $R_{F_1}$ live in the same space modulo $\mf{q}_{43}$, the ideal of norm $43^2$ in $L'$, which satisfies $\mf{q}_{43} \cap \mathcal{O}_{L} = \p_{43}$. As before this just means that $\p_{43}$ is inert in $L'/L$. 

We already have the scaled period polynomials, and reducing them both modulo $\mf{q}_{43}$, we find 
\begin{equation*}
	R_\Delta \equiv 5 R_{F_1} \Mod{\mf{q}_{43}},
\end{equation*}
and so 
\begin{equation*}
	\Delta \equiv F_1 \Mod{\p_{43}},
\end{equation*}
as indicated by the copy of 43 appearing in the pairing in (\ref{pairing_Delta_Delta}). 

The fact that 43 appears both in the pairing and as the modulus of a congruence here \textit{could} be a coincidence, so we compute the same pairing, this time with $F_1$. Again computing, we have
\begin{equation*}
	N(\langle r_{F_1},v_{F_1} \rangle) = 2^2 \cdot 5^4 \cdot 7^4 \cdot 13^4 \cdot 43^4\cdot 173^2
\end{equation*}
We again discard the primes less than the weight $10$. We again see 43 appearing, suggesting this construction does capture congruence moduli. We also note that 
\begin{equation*}
	F_1 \equiv F_2 \Mod{\p},
\end{equation*}
where $\p \in \{ \p_{11}, \p_{43}, \p_{173} \}$, since $F_1$ and $F_2$ are $L$-conjugate, and each of $\{11,43,173\}$ ramifies in $L$. As they come from the structure of the coefficient field $L$, these congruences are much less interesting the others.

Meanwhile, the reappearance of $13$ in the product is, again, mysterious, and we cannot account for it. We note that it does \textit{not} come from a congruence with a torsion class mod 13, as there is no 13-torsion in $H^2(\Gamma_{11},V_{10,10}(\mathcal{O}_{11}))$, per Table 13 of \cite{Sengun_exp}.



\begin{thebibliography}{STF}

\bibitem{Ash_94} A. Ash. \emph{Unstable cohomology of ${\rm SL}(n,\mathcal{O})$}, J. Algebra 167 (1994) 330--342. 

\bibitem{Borel_Serre_corners} A .Borel and J.P. Serre. \emph{Corners and arithmetic groups}. Comment. Math. Helv. 48 (1973) 436--491. 

\bibitem{Magma} W. Bosma, J. Cannon and C. Playoust. \emph{The Magma Algebra System I: The User Language}. Journal of Symbolic Computation, 24 (1997), 235--265. 

\bibitem{Cremona-book} J. E. Cremona. \emph{Algorithms for modular elliptic curves}. Cambridge University Press, USA. 1993. 

\bibitem{Cremona_Whitley} J. Cremona and E. Whitley. \emph{Periods of Cusp forms and elliptic curves over imaginary quadratic fields}. Mathematics of Computation 62 (1994): 407-429.

\bibitem{Fine_1989} B. Fine. \emph{Algebraic theory of the Bianchi groups}. Monographs and Textbooks in Pure and Applied Mathematics, 129. Marcel Dekker, Inc., New York, 1989.

\bibitem{Finis_etal} T. Finis, F. Grunewald, and P. Tirao. \emph{Cohomology of Lattices in ${\rm SL}_2(\C)$}. Exp. Math. 19 (2010), 29-63

\bibitem{Gaba_Popa} R. Gaba and A. A. Popa. \emph{A generalisation of Ramanujan's congruence to modular forms of prime level.} J. of Number Theory 193 (2018), 48-73.

\bibitem{Haberland} H. Haberland. \emph{Perioden von Modulformen einer Variablen und Gruppenkohomologie, I}. Math. Nachr. 112 (1983) 245--282.

\bibitem{Hida_94} H. Hida. \emph{On the critical values of $L$-functions of $GL(2)$ and $GL(2) \times GL(2)$}. Duke Mathematical Journal, 74(2) (1994), 431--529.

\bibitem{Karabulut} C. Karabulut. \emph{From binary Hermitian forms to parabolic cocycles of Euclidean Bianchi groups}. J. Number Theory 236 (2022), 71--115.

\bibitem{Kohnen_Zagier_1984} W. Kohnen and D. Zagier. \emph{Modular forms with rational periods}. in: R.A. Rankin (Ed.), Modular Forms, Ellis Horwood, Chichechester, 1984, pp. 197--249.

\bibitem{LMFDB} The LMFDB Collaboration, The L-functions and modular forms database, \href{https://www.lmfdb.org}{\texttt{https://www.lmfdb.org}} 2023. [Online; accessed 1 June 2023].

\bibitem{Manin_1973} Y. Manin. \emph{Periods of parabolic forms and p-adic Hecke series}. Mat. Sb. (N.S.), 92(134):3(11) (1973), 378–401; Math. USSR-Sb., 21:3 (1973), 371–393

\bibitem{Mohamed} A. Mohamed. \emph{Universal Hecke $L$-series associated with cuspidal eigenforms over imaginary quadratic fields}. in Computations with modular forms, 225--256, Contrib. Math. Comput. Sci., 6, Springer, Cham, 2014.

\bibitem{Pasol-Popa} V. Pa\c{s}ol and  A. Popa. \emph{Modular forms and period polynomials}. Proc. Lond. Math. Soc. (3) 107 (4) (2013) 713--743.

\bibitem{Rahm_Sengun} A.D. Rahm and M.H. \c{S}eng\"{u}n. \emph{On level one cuspidal Bianchi modular forms}. LMS Journal of Computation and Mathematics, 16 . (2013) 187-199. 

\bibitem{Schwermer_1983} J. Schwermer and K. Vogtmann. \emph{The integral homology of ${\rm SL}_2$ and ${\rm PSL}_2$ of Euclidean imaginary quadratic integers}.
Comment. Math. Helv. 58 (1983), no. 4, 573--598.

\bibitem{Sengun_exp} M.H. \c{S}eng\"{u}n. \emph{On the integral cohomology of Bianchi groups}. Exp. Math. 20 (2011), no. 4, 487--505.

\bibitem{Torrey_12} R. Torrey. \emph{On Serre's conjecture over imaginary quadratic fields}. J. Number Theory 132 (2012), no. 4, 637--656. 

\bibitem{Urban_95} E. Urban. \emph{Formes automorphes cuspidales pour GL(2) sur un corps quadratique imaginaire: valeurs sp\'eciales de fonctions $L$ et congruences}. Compositio Math. 99:3 (1995), 283--324. 

\bibitem{Williams_17} C. Williams. \emph{$p$-adic $L$-functions of Bianchi modular forms}. Proc. Lond. Math. Soc. (3) 114 (2017), no. 4, 614--656.



\end{thebibliography}
\end{document}